\documentclass[letterpaper,10pt]{article}

\usepackage[margin=0.5in]{geometry}

\usepackage{fullpage}
\usepackage{enumerate}
\usepackage{authblk}
\usepackage[colorinlistoftodos]{todonotes}
\usepackage{verbatim}
\usepackage{section, amsthm, textcase, setspace, amssymb, lineno, amsmath, amssymb, amsfonts, latexsym, fancyhdr, longtable, ulem, mathtools}
\usepackage{epsfig, graphicx, pstricks,pst-grad,pst-text,tikz, tkz-berge,colortbl}
\usepackage{epsf}
\usepackage{graphicx, color}
\usepackage{float}
\usepackage[rflt]{floatflt}
\usepackage{amsfonts}
\usepackage{latexsym,enumitem}
\usetikzlibrary{fit,matrix,positioning}
\usepackage{pdflscape}
\usetikzlibrary{decorations.pathreplacing}
\usepackage{kbordermatrix}
\usepackage[most]{tcolorbox}
\usepackage{lipsum}

\usepackage{mathrsfs}

\definecolor{darkgreen}{rgb}{0, 0.5, 0}

\DeclareMathOperator{\rank}{rank}

\newtheorem{theorem}{Theorem}
\newtheorem{lemma}{Lemma}
\newtheorem{corollary}{Corollary}

\newtheorem{definition}{Definition}

\newtheorem{Ex}{Example}

\newtheorem*{theorem*}{Theorem}
\newtheorem{remark}{Remark}

\newcommand{\ind}{{\rm ind \hspace{.1cm}}}

\begin{document}

\title{The index of nilpotent Lie poset algebras}

\author[*]{Vincent E. Coll, Jr.}
\author[*]{Nicholas W. Mayers}
\author[*]{Nicholas V. Russoniello}

\affil[*]{Department of Mathematics, Lehigh University, Bethlehem, PA, 18015}

\maketitle

\begin{abstract}
\noindent
We establish combinatorial formulas for the index of a class of matrix Lie algebras whose matrix forms are encoded by strict partial orderings.
\end{abstract}



\section{Introduction}

The index of a Lie algebra is an important algebraic invariant which was first introduced by Dixmier in 1974 (see \textbf{\cite{D}}). Topical research has concentrated on establishing combinatorial formulas for the index of 
Lie algebras in certain combinatorially defined families. In particular, 
the primary focus has been on the families of seaweed algebras (see \textbf{\cite{DK}}) and Lie poset algebras (see \textbf{\cite{CM}}).\footnote{Similar combinatorial investigations have found success considering the extensions of seaweed and Lie poset algebras to the classical families of Lie algebras. For seaweeds see \textbf{\cite{CHM,Coll2,Elash,Joseph,Panyushev1,Panyushev2,Panyushev3}} and for Lie poset algebras see \textbf{\cite{CMBCD,CM}}.} Each of these families of Lie algebras can be reckoned as matrix algebras defined by a matrix form which is encoded by a combinatorial object.
In the case of seaweed algebras, matrix forms are encoded by pairs of compositions, while for Lie poset algebras, the matrix form is encoded by posets (non-strict partial orderings). Here, our focus is the index theory of matrix algebras whose matrix forms are encoded by strict partial orderings.

\bigskip
Formally, the \textit{index} of a Lie algebra $\mathfrak{g}$ is defined as 
\[\ind \mathfrak{g}=\min_{F\in \mathfrak{g^*}} \dim  (\ker (B_F)),\]

\noindent where $B_F$ is the skew-symmetric \textit{Kirillov form} defined by $B_F(x,y)=F([x,y])$, for all $x,y\in\mathfrak{g}$. Of particular interest are those Lie algebras which have index zero, and are called \textit{Frobenius}.\footnote{Frobenius algebras are of  special interest in deformation and quantum group theory stemming from their connection with the classical Yang-Baxter equation (see \textbf{\cite{G1,G2}}).}

As noted, Lie poset algebras form a class of algebras whose index theory has been investigated. Such algebras can be defined as the Lie algebras naturally arising from the incidence algebras of posets \textbf{\cite{Ro}}. As a result, for each poset $(\mathcal{P},\preceq_{\mathcal{P}})$ with $\mathcal{P}=\{1,\hdots,n\}$, one obtains a Lie algebra $\mathfrak{g}(\mathcal{P})$ consisting of $|\mathcal{P}|\times|\mathcal{P}|$ matrices whose $i,j$-entry can be nonzero if and only if $i\preceq_{\mathcal{P}} j$; the Lie bracket of $\mathfrak{g}(\mathcal{P})$ is given by $[X,Y]=XY-YX$, where juxtaposition denotes standard matrix multiplication. Removing diagonal elements from $\mathfrak{g}(\mathcal{P})$ results in a nilpotent subalgebra which, following \textbf{\cite{nilco}}, we denote by $\mathfrak{g}^{\prec}(\mathcal{P})$ and refer to as a ``nilpotent Lie poset algebra."\footnote{In \textbf{\cite{nilco}}, the authors consider the homology and cohomology of nilpotent Lie poset algebras -- but not their index.}

Here, we consider the index theory of nilpotent Lie poset algebras. In particular, we establish a combinatorial formula for the index of nilpotent Lie poset algebras (see Section~\ref{sec:formulas}). It is worth mentioning that Panov \textbf{\cite{Panov}} develops a mechanism for computing the index of nilpotent Lie poset algebras corresponding to a 
disjoint sum of chains.  He does not develop ``closed-form" formulas as we do here.

We begin with poset preliminaries in Section~\ref{sec:prelim} and conclude with an Epilogue in Section~\ref{sec:conc} where we contrast the results here with recent results in the case of (solvable) Lie poset algebras.

\section{Preliminaries}\label{sec:prelim}

A \textit{finite poset} $(\mathcal{P}, \preceq_{\mathcal{P}})$ consists of a finite set $\mathcal{P}=\{1,\hdots,n\}$ together with a binary relation $\preceq_{\mathcal{P}}$ which is reflexive, anti-symmetric, and transitive. It is further assumed that if $x\preceq_{\mathcal{P}}y$ for $x,y\in\mathcal{P}$, then $x\le y$, where $\le$ denotes the natural ordering on $\mathbb{Z}$. When no confusion will arise, we simply denote a poset $(\mathcal{P}, \preceq_{\mathcal{P}})$ by $\mathcal{P}$, and $\preceq_{\mathcal{P}}$ by $\preceq$. 

Let $p_1,p_2\in\mathcal{P}$. If $p_1\preceq p_2$ and $p_1\neq p_2$, then we call $p_1\preceq p_2$ a \textit{strict relation} and write $p_1\prec p_2$. Let $Rel(\mathcal{P})$ denote the set of strict relations between elements of $\mathcal{P}$, $Ext(\mathcal{P})$ denote the set of minimal and maximal elements of $\mathcal{P}$, and $Rel_E(\mathcal{P})$ denote the number of strict relations between the elements of $Ext(\mathcal{P})$.

\begin{Ex}\label{ex:posnot}
Let $\mathcal{P}$ be the poset $\mathcal{P}=\{1,2,3,4,5,6\}$ with $1,2\preceq 3\preceq 4,5,6$. We have $$Rel(\mathcal{P})=\{1\prec 3,1\prec 4,1\prec 5,1\prec 6,2\prec 3,2\prec 4,2\prec 5,2\prec 6,3\prec 4,3\prec 5,3\prec 6\},$$  $$Ext(\mathcal{P})=\{1,2,4,5,6\},\quad \text{and}\quad Rel_E(\mathcal{P})=\{1\prec 4, 1\prec 5, 1\prec 6,2\prec 4, 2\prec 5, 2\prec 6\}.$$
\end{Ex}

\noindent
If $p_1\prec p_2$ and there does not exist $p\in \mathcal{P}$ satisfying $p_1\prec p\prec p_2$, then $p_1\prec p_2$ is a \textit{covering relation}.  Covering relations are used to define a visual representation of $\mathcal{P}$ called the \textit{Hasse diagram} -- a graph whose vertices correspond to elements of $\mathcal{P}$ and whose edges correspond to covering relations (see, for example, Figure~\ref{fig:poset}). A totally ordered subset $S\subset\mathcal{P}$ is called a \textit{chain}. We define the \textit{height} of a poset $\mathcal{P}$ to be one less than the largest cardinality of a chain in $\mathcal{P}$.

\begin{Ex}\label{ex:poset}
Let $\mathcal{P}$ be the poset of Example~\ref{ex:posnot}. In Figure~\ref{fig:poset} we illustrate the Hasse diagram of $\mathcal{P}$.

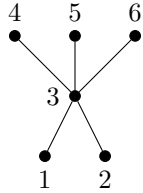
\begin{figure}[H]
$$\begin{tikzpicture}[scale=0.8]
	\node [circle, draw = black, fill = black, inner sep = 0.5mm, label=below:{1}] (1) at (-0.5,0) {};
	\node (2) at (0, 1)[circle, draw = black, fill = black, inner sep = 0.5mm, label=left:{3}] {};
	\node [circle, draw = black, fill = black, inner sep = 0.5mm, label=above:{4}] (3) at (-1,2) {};
	\node [circle, draw = black, fill = black, inner sep = 0.5mm, label=above:{6}] (4) at (1,2) {};
    \draw (1)--(2);
    \draw (2)--(3);
    \draw (2)--(4);
\node (v1) at (0.5,0) [circle, draw = black, fill = black, inner sep = 0.5mm, label=below:{2}] {};
\draw (v1) -- (2);
\node [circle, draw = black, fill = black, inner sep = 0.5mm, label=above:{5}] (v2) at (0,2) {};
\draw (2) -- (v2);
\end{tikzpicture}$$
\caption{Hasse diagram of a poset}\label{fig:poset}
\end{figure}
\end{Ex}

Let \textbf{k} be an algebraically closed field of characteristic zero, which we may take to be the complex numbers. The \textit{nilpotent Lie poset algebra} $\mathfrak{g}^{\prec}(\mathcal{P})=\mathfrak{g}^{\prec}(\mathcal{P}, \textbf{k})$ is the span over $\textbf{k}$ of elements $E_{p_i,p_j}$, for $p_i,p_j\in\mathcal{P}$ satisfying $p_i\prec p_j$, with Lie bracket $[X,Y]=XY-YX$, where $E_{p_i,p_j}E_{p_k,p_l}=E_{p_i,p_l}$ if $p_j=p_k$ and $0$ otherwise. The algebra $\mathfrak{g}^{\prec}(\mathcal{P})$ may be regarded as a subalgebra of the algebra of $n \times n$ strictly upper-triangular matrices over $\textbf{k}$ by replacing each basis element $E_{p_i,p_j}$ by the $n\times n$ matrix containing a 1 in the $i,j$-entry and 0's elsewhere. The product of elements $E_{p_i,p_j}$ becomes matrix multiplication.

\begin{Ex}\label{ex:matrixform}  
Let $\mathcal{P}$ be the poset given in Example~\ref{ex:posnot}. The matrix form of elements in $\mathfrak{g}^{\prec}(\mathcal{P})$ is illustrated in Figure~\ref{fig:matform}, where the $\ast$'s denote potential non-zero entries. 
\begin{figure}[H]
$$\kbordermatrix{
    & 1 & 2 & 3 & 4 & 5 & 6  \\
   1 & 0 & 0 & * & *  & * & *   \\
   2 & 0 & 0 & * & * & * & *   \\
   3 & 0 & 0 & 0 & * & * & *   \\
   4 & 0 & 0 & 0 & 0 & 0 & 0   \\
   5 & 0 & 0 & 0 & 0 & 0 & 0   \\
   6 & 0 & 0 & 0 & 0 & 0 & 0   \\
  }$$
\caption{Matrix form of $\mathfrak{g}^{\prec}(\mathcal{P})$, for $\mathcal{P}=\{1,2,3,4,5,6\}$ with $1,2\preceq3\preceq4,5,6$}\label{fig:matform}
\end{figure}
\end{Ex}

\section{Combinatorial index formulas}\label{sec:formulas}

In this section, we determine combinatorial formulas for the index of nilpotent Lie poset algebras.

It will be convenient to use an alternative characterization of the index. Let $\mathfrak{g}$ be an arbitrary Lie algebra with basis $\{x_1,...,x_n\}$.  The index of $\mathfrak{g}$ can be expressed using the \textit{commutator matrix}, $([x_i,x_j])_{1\le i, j\le n}$, over the quotient field $R(\mathfrak{g})$ of the symmetric algebra $Sym(\mathfrak{g})$ as follows (see \textbf{\cite{D}}).

\begin{theorem}\label{thm:commat}
 The index of $\mathfrak{g}$ is given by  
 $$\ind \mathfrak{g}= n-\rank_{R(\mathfrak{g})}([x_i,x_j])_{1\le i, j\le n}.$$ 
\end{theorem}

\begin{Ex}
Consider the Lie algebra $\mathfrak{g}$ consisting of the upper triangular matrices in $\mathfrak{sl}(2)$. A Chevalley basis for $\mathfrak{g}$ is given by $\{x_1,x_2\}$, where $[x_1,x_2]=2x_2$. The commutator matrix of $\mathfrak{g}$ is illustrated in Figure~\ref{ex:commatsl2}.  Since the rank of this matrix is two, it follows from Theorem~\ref{thm:commat} that $\mathfrak{g}$ is Frobenius.
\begin{figure}[H]
$$\begin{bmatrix}
   0 & 2x_2  \\
     -2x_2 & 0 
\end{bmatrix}$$
\caption{Commutator matrix}\label{ex:commatsl2}
\end{figure}
\end{Ex}

\begin{remark}
To ease notation, row and column labels of commutator matrices will be bolded and matrix entries will be unbolded. Furthermore, we will refer to the row corresponding to $\mathbf{x}$ in a commutator matrix -- and by a slight abuse of notation, in any equivalent matrix -- as row $\mathbf{x}$.
\end{remark}

Throughout this section, given a poset $\mathcal{P}$, we set $$C(\mathfrak{g}^{\prec}(\mathcal{P}))=([x_i,x_j])_{1\le i, j\le n}\text{, where }\{x_1,\hdots,x_n\}=\{E_{p_i,p_j}:p_i,p_j\in\mathcal{P},p_i\prec p_j\}.$$

\begin{theorem}
If $\mathcal{P}$ is a poset, then $\ind\mathfrak{g}^{\prec}(\mathcal{P})\ge |Rel_E(\mathcal{P})|$.
\end{theorem}
\begin{proof}
Basis elements of the form $E_{p_i,p_j}$, for $p_i,p_j\in Ext(\mathcal{P})$ satisfying $p_i\prec p_j$, commute with all other elements of $\mathfrak{g}^{\prec}(\mathcal{P})$; that is, such elements correspond to zero rows in $\mathcal{C}(\mathfrak{g}^{\prec}(\mathcal{P}))$. The result follows.
\end{proof}

\begin{corollary}\label{cor:nofrob}
There are no Frobenius nilpotent Lie poset algebras.
\end{corollary}

\begin{remark}
Corollary~\ref{cor:nofrob} holds more generally. In particular, in \textup{\textbf{\cite{Frobnil}}} it is shown that no real or complex, finite-dimensional nilpotent Lie algebra can be Frobenius.
\end{remark}

\begin{corollary}\label{cor:h1}
If $\mathcal{P}$ is a height-one poset, then $\ind\mathfrak{g}^{\prec}(\mathcal{P})= |Rel_E(\mathcal{P})|$.
\end{corollary}
\begin{proof}
All basis elements of $\mathfrak{g}^{\prec}(\mathcal{P})$ of the form $E_{p_i,p_j}$ satisfy $p_i,p_j\in Ext(\mathcal{P})$. The result follows.
\end{proof}

The first non-trivial result concerning the index of nilpotent Lie poset algebras corresponds to posets of height two. Before establishing the corresponding index formula, the following notation will prove helpful in the results that follow.

\begin{definition}
If $\mathcal{P}$ is a poset and $p\in \mathcal{P}$, then
$$D(\mathcal{P},p)=|\{q\in \mathcal{P}:q\prec p\}|,$$  $$U(\mathcal{P},p)=|\{q\in \mathcal{P}:p\prec q\}|,$$ $$\mathscr{B}^{p}(\mathcal{P})=\{E_{p,b}\in\mathfrak{g}^{\prec}(\mathcal{P}):b\text{ is maximal in }\mathcal{P}\text{ and }p\prec b\},$$
$$\mathscr{B}_{p}(\mathcal{P})=\{E_{l,p}\in\mathfrak{g}^{\prec}(\mathcal{P}):l\text{ is minimal in }\mathcal{P}\text{ and }l\prec p\},$$ $$U_E(\mathcal{P},p)=|\mathscr{B}^{p}(\mathcal{P})|,\text{ and}$$
$$D_E(\mathcal{P},p)=|\mathscr{B}_{p}(\mathcal{P})|.$$ 
\end{definition}

\begin{theorem}\label{thm:h2ind}
If $\mathcal{P}$ is a height-two poset, then $$\ind\mathfrak{g}^{\prec}(\mathcal{P})=|Rel(\mathcal{P})|-2\sum_{p\in\mathcal{P}\backslash Ext(\mathcal{P})}\min(D(\mathcal{P},p),U(\mathcal{P},p)).$$
\end{theorem}
\begin{proof}
Arrange for the row labels of $C(\mathfrak{g}^{\prec}(\mathcal{P}))$ to be ordered as follows.
\begin{itemize}
    \item First, for $p\in \mathcal{P}\backslash Ext(\mathcal{P})$ in increasing order of $p$ in $\mathbb{Z}$, list the elements of each set $\mathscr{B}_p$ in increasing lexicographic order of their indices in $\mathbb{Z}^2$.
    \item Next, for $p\in \mathcal{P}\backslash Ext(\mathcal{P})$ in increasing order of $p$ in $\mathbb{Z}$, list the elements of each set $\mathscr{B}^p$ in increasing lexicographic order of their indices in $\mathbb{Z}^2$.
    \item Finally, list basis elements of the form $E_{l,b}$, for $l,b\in Ext(\mathcal{P})$ satisfying $l\prec b$, in increasing lexicographic order of $(l,b)$ in $\mathbb{Z}^2$.
\end{itemize}
Similarly, arrange for the columns labels of $C(\mathfrak{g}^{\prec}(\mathcal{P}))$ to be ordered as follows.
\begin{itemize}
    \item First, for $p\in \mathcal{P}\backslash Ext(\mathcal{P})$ in increasing order of $p$ in $\mathbb{Z}$, list the elements of each set $\mathscr{B}^p$ in increasing lexicographic order of their indices in $\mathbb{Z}^2$.
    \item Next, for $p\in \mathcal{P}\backslash Ext(\mathcal{P})$ in increasing order of $p$ in $\mathbb{Z}$, list the elements of each set $\mathscr{B}_p$ in increasing lexicographic order of their indices in $\mathbb{Z}^2$.
    \item Finally, list basis elements of the form $E_{l,b}$, for $l,b\in Ext(\mathcal{P})$ satisfying $l\prec b$, in increasing lexicographic order of $(l,b)$ in $\mathbb{Z}^2$.
\end{itemize}
Assuming the given ordering of row and column labels, $C(\mathfrak{g}^{\prec}(\mathcal{P}))$ is block diagonal with blocks $M_p$ and $-M_p^T$, where $M_p$ (resp., $-M_p$) is defined by row labels coming from $\mathscr{B}_p$ (resp., $\mathscr{B}^p$) and column labels coming from $\mathscr{B}^p$ (resp., $\mathscr{B}_p$). By definition, the entries of $M_p$ (resp., $-M_p^T$) are all distinct and nonzero. Thus, the rank of $M_p$ (resp., $M_p^T$) must be maximal. Since $M_p$ (resp., $M_p^T$) has $D(\mathcal{P},p)$ rows (resp., columns) and $U(\mathcal{P},p)$ columns (resp., rows), the rank of $C(\mathfrak{g}^{\prec}(\mathcal{P}))$ is given by $2\sum_{p\in\mathcal{P}\backslash Ext(\mathcal{P})}\min(D(\mathcal{P},p),U(\mathcal{P},p))$. Applying Theorem~\ref{thm:commat} establishes the result.
\end{proof}

\begin{Ex}
Let $\mathcal{P}$ be the poset given in Example~\ref{ex:posnot}; that is, $\mathcal{P}=\{1,2,3,4,5,6\}$ with $1,2\preceq 3\preceq 4,5,6$. The commutator matrix $C(\mathfrak{g}^{\prec}(\mathcal{P}))$ is illustrated in Figure~\ref{ex:h2commat}, assuming the ordering of the row and column labels as in the proof of Theorem~\ref{thm:h2ind}. Further, note that since $$|Rel(\mathcal{P})|=11,\quad \mathcal{P}\backslash Ext(\mathcal{P})=\{3\},\quad U(\mathcal{P},3)=3,\quad D(\mathcal{P},3)=2,\quad\text{and}\quad min(U(\mathcal{P},3),D(\mathcal{P},3))=2,$$ applying Theorem~\ref{thm:h2ind} shows that $\ind\mathfrak{g}^{\prec}(\mathcal{P})=11-2(2)=7$.
\begin{figure}[H]
$$\kbordermatrix{
    & \mathbf{E_{3,4}} & \mathbf{E_{3,5}} & \mathbf{E_{3,6}} & \mathbf{E_{1,3}} & \mathbf{E_{2,3}} & \mathbf{E_{1,4}} & \mathbf{E_{1,5}} & \mathbf{E_{1,6}} & \mathbf{E_{2,4}} & \mathbf{E_{2,5}} & \mathbf{E_{2,6}}  \\
   \mathbf{E_{1,3}} & E_{1,4} & E_{1,5} & E_{1,6} & 0 & 0 & 0 & 0 & 0 & 0 & 0 & 0   \\
   \mathbf{E_{2,3}} & E_{2,4} & E_{2,5} & E_{2,6} & 0 & 0 & 0 & 0 & 0 & 0 & 0 & 0  \\
   \mathbf{E_{3,4}} & 0 & 0 & 0 & -E_{1,4} & -E_{2,4} & 0 & 0 & 0 & 0 & 0 & 0  \\
   \mathbf{E_{3,5}} & 0 & 0 & 0 & -E_{1,5} & -E_{2,5} & 0 & 0 & 0 & 0 & 0 & 0  \\
   \mathbf{E_{3,6}} & 0 & 0 & 0 & -E_{1,6} & -E_{2,6} & 0 & 0 & 0 & 0 & 0 & 0  \\
   \mathbf{E_{1,4}} & 0 & 0 & 0 & 0 & 0 & 0 & 0 & 0  & 0 & 0 & 0  \\
   \mathbf{E_{1,5}} & 0 & 0 & 0 & 0 & 0 & 0 & 0 & 0  & 0 & 0 & 0  \\
   \mathbf{E_{1,6}} & 0 & 0 & 0 & 0 & 0 & 0 & 0 & 0  & 0 & 0 & 0  \\
   \mathbf{E_{2,4}} & 0 & 0 & 0 & 0 & 0 & 0 & 0 & 0  & 0 & 0 & 0  \\
   \mathbf{E_{2,5}} & 0 & 0 & 0 & 0 & 0 & 0 & 0 & 0  & 0 & 0 & 0  \\
   \mathbf{E_{2,6}} & 0 & 0 & 0 & 0 & 0 & 0 & 0 & 0  & 0 & 0 & 0  \\
  }$$
\caption{$C(\mathfrak{g}^{\prec}(\mathcal{P}))$}\label{ex:h2commat}
\end{figure}
\end{Ex}

In Theorem~\ref{thm:generalind} below, we determine a height-independent formula for the index of nilpotent Lie poset algebras. This is done through the use of an inductive argument with Corollary~\ref{cor:h1} and Theorem~\ref{thm:h2ind} covering the first two base cases. In the proof, given a height $n$ poset $\mathcal{P}$, we are able to recursively construct a height $n-1$ poset $\mathcal{P}_m$ for which $\rank\text{ }C(\mathfrak{g}^{\prec}(\mathcal{P}))=\rank\text{ }C(\mathfrak{g}^{\prec}(\mathcal{P}_m))$. Of particular importance in the argument is the set of ``middle" sections of maximal chains, which we make precise as follows: if $\mathcal{P}$ is a height-$n$ poset, then $$\mathscr{M}_n(\mathcal{P})=\{\{p_1\prec\hdots\prec p_{n-1}\}\subset\mathcal{P}:\exists~p_0,p_n\in\mathcal{P}\text{ such that }p_0\prec p_1\prec\hdots\prec p_{n-1}\prec p_n\}.$$
\noindent
Further, the construction of $\mathcal{P}_m$ from $\mathcal{P}$ is guided by row operations on the commutator matrix $C(\mathfrak{g}^{\prec}(\mathcal{P}))$ which correspond to splitting a row $r$ of $C(\mathfrak{g}^{\prec}(\mathcal{P}))$ into two separate rows $r_1$ and $r_2$ satisfying $r_1+r_2=r$; to aid discourse of such row operations, we make the following definition.

\begin{definition}
Let $M$ be an $n\times n$ matrix with row and column labels $L=\{l_1,\hdots,l_n\}$, and let $r=(r_{l_1},\hdots,r_{l_n})$ be a row of $M$. We define the restriction of $r$ to columns $S\subset L$ to be the vector $r_S=(r'_{l_1},\hdots,r'_{l_n})$ with $r'_{l_i}=r_{l_i}$ if $l_i\in S$, and $r'_{l_i}=0$ otherwise.
\end{definition}

\begin{theorem}\label{thm:generalind}
If $\mathcal{P}$ is a poset, then
\begin{equation}\label{eqn:mainthm}
\ind\mathfrak{g}^{\prec}(\mathcal{P})=|Rel(\mathcal{P})|-2\sum_{p\in\mathcal{P}\backslash Ext(\mathcal{P})}\min(D(\mathcal{P},p),U(\mathcal{P},p)).
\end{equation}
\end{theorem}
\begin{proof}
By induction on the height of $\mathcal{P}$. The result holds for height-one and height-two posets, as shown in Corollary~\ref{cor:h1} and Theorem~\ref{thm:h2ind}, respectively. Assume the result holds for posets of height at most $n-1\ge 2$, let $\mathcal{P}$ be a poset of height $n$, and let $$\mathscr{B}(\mathcal{P})=\{E_{p_i,p_j}:p_i,p_j\in\mathcal{P},p_i\prec p_j\}.$$ Note then that there exists $\{p_1\prec\hdots\prec p_{n-1}\}\in\mathscr{M}_n(\mathcal{P})$. Further, the rows of $C(\mathfrak{g}^{\prec}(\mathcal{P}))$ labeled by elements from the set $\mathscr{B}_{p_{n-1}}(\mathcal{P})$ are equal to their restrictions to columns labeled by elements of $\mathscr{B}^{p_{n-1}}(\mathcal{P})$. Similarly, the rows of $C(\mathfrak{g}^{\prec}(\mathcal{P}))$ labeled by elements from the set $\mathscr{B}^{p_{1}}(\mathcal{P})$ are equal to their restrictions to columns labeled by elements of $\mathscr{B}_{p_{1}}(\mathcal{P})$. Now, denote by $M_{p_{n-1}}$ the submatrix of $C(\mathfrak{g}^{\prec}(\mathcal{P}))$ defined by rows labeled by elements of $\mathscr{B}_{p_{n-1}}(\mathcal{P})$ and columns labeled by elements of $\mathscr{B}^{p_{n-1}}(\mathcal{P}).$ Similarly, denote by $M_{p_1}$ the submatrix of $C(\mathfrak{g}^{\prec}(\mathcal{P}))$ defined by rows labeled by elements of $\mathscr{B}^{p_{1}}(\mathcal{P})$ and columns labeled by elements of $\mathscr{B}_{p_{1}}(\mathcal{P})$. Thus, $C(\mathfrak{g}^{\prec}(\mathcal{P}))$ contains the $D_E(\mathcal{P},p_{n-1})\times U_E(\mathcal{P},p_{n-1})$ submatrix $M_{p_{n-1}}$ and the $U_E(\mathcal{P},p_1)\times D_E(\mathcal{P},p_1)$ submatrix $M_{p_1}$; moreover, each of $M_{p_{n-1}}$ and $M_{p_1}$ has all entries nonzero and pairwise unequal. The proof now breaks into two cases.
\\*

\noindent
\textbf{Case 1}: $D_E(\mathcal{P},p_{n-1})\ge U_E(\mathcal{P},p_{n-1})$. Now note that the columns of $C(\mathfrak{g}^{\prec}(\mathcal{P}))$ may be organized so that the collection of rows indexed by elements from the set $\mathscr{B}_{p_{n-1}}(\mathcal{P})$ takes the form 
$$R=\begin{bmatrix}
\huge{M_{p_{n-1}}} & \huge{0} & \huge{0} & \dots & \huge{0}
\end{bmatrix}.$$ Since $M_{p_{n-1}}$ occupies exactly the columns indexed by elements from the set $\mathscr{B}^{p_{n-1}}(\mathcal{P})$ and has maximal rank, $(\ast)$ the restriction of all rows in $C(\mathfrak{g}^{\prec}(\mathcal{P}))$ to columns indexed by elements in $\mathscr{B}^{p_{n-1}}(\mathcal{P})$ is spanned by $R.$ Also notice that the only rows with nonzero entries in columns labeled by elements of $\mathscr{B}^{p_{n-1}}(\mathcal{P})$ are $\mathbf{E_{q,p_{n-1}}}$, where $q\in\mathcal{P}$ satisfies $q\prec p_{n-1}$. Now, consider the subset of those rows $\mathbf{E_{q,p_{n-1}}}$ with $q\prec p_{n-1}$ and $q\in\mathcal{P}\setminus Ext(\mathcal{P}).$ Such rows only have nonzero entries in columns with labels from the set $\mathscr{B}^{p_{n-1}}(\mathcal{P})$ and labels of the form $\mathbf{E_{r,q}}$, for $r\in\mathcal{P}$ satisfying $r\prec q\prec p_{n-1}$. By $(\ast)$, the rank of $C(\mathfrak{g}^{\prec}(\mathcal{P}))$ is unaltered by splitting such rows into two rows, one corresponding to the restriction to columns labeled by $\mathscr{B}^{p_{n-1}}(\mathcal{P})$ (denote by $\mathbf{E_{q_{p_{n-1}},p_{n-1}}}$), and the other corresponding to the restriction to columns labeled by elements of $\mathscr{B}(\mathcal{P})\backslash \mathscr{B}^{p_{n-1}}(\mathcal{P})$ (denote by $\mathbf{E_{q,p'_{n-1}}}$). Additionally, since row $\mathbf{E_{q_{p_{n-1}},p_{n-1}}}$ is in the span of $R$, the rank is unaffected by replacing labels $E_{q,r},$ for $q\prec p_{n-1}\prec r$, in row $\mathbf{E_{q_{p_{n-1}},p_{n-1}}}$ with $E_{q_{p_{n-1}},r}.$ Further, as there are no two nonzero equal entries in any row \textup(resp., column\textup) of $C(\mathfrak{g}^{\prec}(\mathcal{P}))$, the rank is also not altered by relabeling the entries $-E_{r,p_{n-1}}$, for $r\prec q\prec p_{n-1}$, of row $\mathbf{E_{q,p'_{n-1}}}$ by $-E_{r,p'_{n-1}}$. Now, simply take the transpose of $C(\mathfrak{g}^{\prec}(\mathcal{P}))$, multiply by $-1$, and perform the same row operations; note, this is well-defined since the rows altered/added in the above algorithm contribute entries of zero to the rows of interest in the transpose. Removing zero rows and columns, the resulting matrix is $C(\mathfrak{g}^{\prec}(\mathcal{P}_1))$ (with zero rows and columns removed), where the poset $\mathcal{P}_1$ is constructed from $\mathcal{P}$ as follows: 
\begin{itemize}
    \item for each $q\in\mathcal{P}\setminus Ext(\mathcal{P})$ satisfying $q\prec p_{n-1}$ remove the relation $q\prec p_{n-1}$ and add a new minimal element $q_{p_{n-1}}\prec p_{n-1}$; and
    \item add a new maximal element $p_{n-1}'$ which satisfies $q\prec p_{n-1}'$, for all $q\in\mathcal{P}$ such that $q\prec_{\mathcal{P}} p_{n-1}$.
\end{itemize}
By construction, $\rank C(\mathfrak{g}^{\prec}(\mathcal{P}))=\rank C(\mathfrak{g}^{\prec}(\mathcal{P}_1)).$ Note also that $\mathcal{P}\setminus Ext(\mathcal{P})=\mathcal{P}_1\setminus Ext(\mathcal{P}_1),$ 
\begin{equation}\label{eqn:c1}
    U(\mathcal{P},p)=U(\mathcal{P}_1,p)\text{ and }D(\mathcal{P},p)=D(\mathcal{P}_1,p)\text{, for all }p\in\mathcal{P}\backslash Ext(\mathcal{P}),
\end{equation}
and $|\mathscr{M}_n(\mathcal{P}_1)|<|\mathscr{M}_n(\mathcal{P})|$.
\\*

\noindent
\textbf{Case 2}: $D_E(\mathcal{P},p_{n-1})< U_E(\mathcal{P},p_{n-1})$. In this case, $$U_E(\mathcal{P},p_1)\ge U_E(\mathcal{P},p_{n-1})>D_E(\mathcal{P},p_{n-1})\ge 
D_E(\mathcal{P},p_1).$$ Note that the columns of $C(\mathfrak{g}^{\prec}(\mathcal{P}))$ may be organized so that the collection of rows indexed by elements from the set $\mathscr{B}^{p_1}(\mathcal{P})$ takes the form 
$$R'=\begin{bmatrix}
\huge{M_{p_1}} & \huge{0} & \huge{0} & \dots & \huge{0}
\end{bmatrix}.$$ Since $M_{p_1}$ occupies exactly the columns indexed by elements from the set $\mathscr{B}_{p_1}(\mathcal{P})$ and has maximal rank, $(\ast\ast)$ the restriction of all rows in $C(\mathfrak{g}^{\prec}(\mathcal{P}))$ to columns indexed by elements in $\mathscr{B}_{p_1}(\mathcal{P})$ is spanned by $R'.$ Also notice that the only rows with nonzero entries in columns labeled by elements of $\mathscr{B}_{p_1}(\mathcal{P})$ are $\mathbf{E_{p_1,q}}$, where $q\in\mathcal{P}$ satisfies $p_1\prec q$. Now, consider the subset of those rows $\mathbf{E_{p_1,q}}$ with $p_1\prec q$ and $q\in\mathcal{P}\setminus Ext(\mathcal{P}).$ Such rows only have nonzero entries in columns with labels from the set $\mathscr{B}_{p_1}(\mathcal{P})$ and labels of the form $\mathbf{E_{q,r}}$, for $r\in\mathcal{P}$ satisfying $p_1\prec q\prec r$. By $(\ast\ast)$, the rank of $C(\mathfrak{g}^{\prec}(\mathcal{P}))$ is unaltered by splitting such rows into two rows, one corresponding to the restriction to columns labeled by $\mathscr{B}_{p_1}(\mathcal{P})$ (denote by $\mathbf{E_{p_1,q^{p_1}}}$), and the other corresponding to the restriction to columns labeled by elements of $\mathscr{B}(\mathcal{P})\backslash \mathscr{B}_{p_1}(\mathcal{P})$ (denote by $\mathbf{E_{p'_1,q}}$). Additionally, since row $\mathbf{E_{p_1,q^{p_1}}}$ is in the span of $R$, the rank is unaffected by replacing labels $-E_{r,q},$ for $r\prec p_1\prec q$, in row $\mathbf{E_{p_1,q^{p_1}}}$ with $-E_{r,q^{p_1}}.$ Further, as there are no two nonzero equal entries in any row \textup(resp., column\textup) of $C(\mathfrak{g}^{\prec}(\mathcal{P}))$, the rank is also not altered by relabeling the entries $E_{p_1, r}$, for $p_1\prec q\prec r$, of row $\mathbf{E_{p'_1,q}}$ by $E_{p'_1,r}$. Now, simply take the transpose of $C(\mathfrak{g}^{\prec}(\mathcal{P}))$, multiply by $-1$, and perform the same row operations; note, this is well-defined since the rows altered/added in the above algorithm contribute entries of zero to the rows of interest in the transpose. Removing zero rows and columns, the resulting matrix is $C(\mathfrak{g}^{\prec}(\mathcal{P}_1))$ (with zero rows and columns removed), where the poset $\mathcal{P}_1$ is constructed from $\mathcal{P}$ as follows: 
\begin{itemize}
    \item for each $q\in\mathcal{P}\setminus Ext(\mathcal{P})$ satisfying $p_1\prec q$ remove the relation $p_1\prec q$ and add a new maximal element $p_1\prec q^{p_1}$; and
    \item add a new minimal element $p'_1$ which satisfies $p'_1\prec q$, for all $q\in\mathcal{P}$ such that $p_1\prec_{\mathcal{P}} q$.
\end{itemize}
By construction, $\rank C(\mathfrak{g}^{\prec}(\mathcal{P}))=\rank C(\mathfrak{g}^{\prec}(\mathcal{P}_1)).$ Note also that $\mathcal{P}\setminus Ext(\mathcal{P})=\mathcal{P}_1\setminus Ext(\mathcal{P}_1),$ 
\begin{equation}\label{eqn:c2}
    U(\mathcal{P},p)=U(\mathcal{P}_1,p)\text{ and }D(\mathcal{P},p)=D(\mathcal{P}_1,p)\text{, for all }p\in\mathcal{P}\backslash Ext(\mathcal{P}),
\end{equation}
and $|\mathscr{M}_n(\mathcal{P}_1)|<|\mathscr{M}_n(\mathcal{P})|$.
\\*

\noindent
Since $\mathcal{P}$ is finite, one can perform the above operations inductively until one arrives at a poset $\mathcal{P}_m$ satisfying $\mathscr{M}_n(\mathcal{P}_m)=\emptyset$; that is, $\mathcal{P}_m$ is of height $<n$. Further, $$rank\text{ }C(\mathfrak{g}^{\prec}(\mathcal{P}))=rank\text{ }C(\mathfrak{g}^{\prec}(\mathcal{P}_m)).$$ Considering $\mathcal{P}\backslash Ext(\mathcal{P})=\mathcal{P}_m\backslash Ext(\mathcal{P}_m)$, (\ref{eqn:c1}), and (\ref{eqn:c2}), the result follows.
\end{proof}

\begin{Ex}
Let $\mathcal{P}=\{1,2,3,4,5,6,7\}$ with $1,2\prec 3\prec 5\prec 6,7$ and $2\prec 4\prec 7$. In Figure~\ref{fig:matpremap} we illustrate the commutator matrix $C(\mathfrak{g}^{\prec}(\mathcal{P}))$ with zero rows and columns removed.

\begin{figure}[H]
$$\kbordermatrix{
    & \mathbf{E_{5,6}} & \mathbf{E_{5,7}} & \mathbf{E_{1,3}} & \mathbf{E_{2,3}} & \mathbf{E_{4,7}} & \mathbf{E_{2,4}} & \mathbf{E_{3,5}} & \mathbf{E_{1,5}} & \mathbf{E_{2,5}} & \mathbf{E_{3,6}} & \mathbf{E_{3,7}}   \\
   \mathbf{E_{1,5}} & E_{1,6} & E_{1,7} & 0 & 0 & 0 & 0 & 0 & 0 & 0 & 0 & 0    \\
   \mathbf{E_{2,5}} & E_{2,6} & E_{2,7} & 0 & 0 & 0 & 0 & 0 & 0 & 0 & 0 & 0   \\
   \mathbf{E_{3,6}} & 0 & 0 & -E_{1,6} & -E_{2,6} & 0 & 0 & 0 & 0 & 0 & 0 & 0  \\
   \mathbf{E_{3,7}} & 0 & 0 & -E_{1,7} & -E_{2,7} & 0 & 0 & 0 & 0 & 0 & 0 & 0   \\
   \mathbf{E_{3,5}} & E_{3,6} & E_{3,7} & -E_{1,5} & -E_{2,5} & 0 & 0 & 0 & 0 & 0 & 0 & 0   \\
   \mathbf{E_{2,4}} & 0 & 0 & 0 & 0 & E_{2,7} & 0 & 0 & 0 & 0 & 0 & 0   \\
   \mathbf{E_{4,7}} & 0 & 0 & 0 & 0 & 0 & -E_{2,7} & 0 & 0 & 0 & 0 & 0   \\
   \mathbf{E_{5,6}} & 0 & 0 & 0 & 0 & 0 & 0 & -E_{3,6} & -E_{1,6} & -E_{2,6} & 0 & 0   \\
   \mathbf{E_{5,7}} & 0 & 0 & 0 & 0 & 0 & 0 & -E_{3,7} & -E_{1,7} & -E_{2,7} & 0 & 0   \\
   \mathbf{E_{1,3}} & 0 & 0 & 0 & 0 & 0 & 0 & E_{1,5} & 0 & 0 & E_{1,6} & E_{1,7}   \\
   \mathbf{E_{2,3}} & 0 & 0 & 0 & 0 & 0 & 0 & E_{2,5} & 0 & 0 & E_{2,6} & E_{2,7}  \\
  }$$
\caption{$C(\mathfrak{g}^{\prec}(\mathcal{P}))$}\label{fig:matpremap}
\end{figure}

\noindent
In Figure~\ref{fig:matpostmap} we illustrate the matrix resulting from the recursive procedure outlined in the proof of Theorem~\ref{thm:generalind} with zero rows and columns removed; that is, the commutator matrix $C(\mathfrak{g}^{\prec}(\mathcal{P}_1))$ with zero rows and columns removed.

\begin{figure}[H]
$$\kbordermatrix{
    & \mathbf{E_{5,6}} & \mathbf{E_{5,7}} & \mathbf{E_{1,3}} & \mathbf{E_{2,3}} & \mathbf{E_{4,7}} & \mathbf{E_{2,4}} & \mathbf{E_{5_3,5}} & \mathbf{E_{3,5'}} & \mathbf{E_{1,5}} & \mathbf{E_{2,5}} & \mathbf{E_{3,6}} & \mathbf{E_{3,7}}   \\
   \mathbf{E_{1,5}} & E_{1,6} & E_{1,7} & 0 & 0 & 0 & 0 & 0 & 0 & 0 & 0 & 0 & 0    \\
   \mathbf{E_{2,5}} & E_{2,6} & E_{2,7} & 0 & 0 & 0 & 0 & 0 & 0 & 0 & 0 & 0 & 0   \\
   \mathbf{E_{3,6}} & 0 & 0 & -E_{1,6} & -E_{2,6} & 0 & 0 & 0 & 0 & 0 & 0 & 0 & 0  \\
   \mathbf{E_{3,7}} & 0 & 0 & -E_{1,7} & -E_{2,7} & 0 & 0 & 0 & 0 & 0 & 0 & 0 & 0   \\
   \mathbf{E_{5_3,5}} & E_{5_3,6} & E_{5_3,7} & 0 & 0 & 0 & 0 & 0 & 0 & 0 & 0 & 0 & 0   \\
   \mathbf{E_{3,5'}} & 0 & 0 & -E_{1,5'} & -E_{2,5'} & 0 & 0 & 0 & 0 & 0 & 0 & 0 & 0  \\
   \mathbf{E_{2,4}} & 0 & 0 & 0 & 0 & E_{2,7} & 0 & 0 & 0 & 0 & 0 & 0 & 0   \\
   \mathbf{E_{4,7}} & 0 & 0 & 0 & 0 & 0 & -E_{2,7} & 0 & 0 & 0 & 0 & 0 & 0   \\
   \mathbf{E_{5,6}} & 0 & 0 & 0 & 0 & 0 & 0 & -E_{5_3,6} & 0 & -E_{1,6} & -E_{2,6} & 0 & 0   \\
   \mathbf{E_{5,7}} & 0 & 0 & 0 & 0 & 0 & 0 & -E_{5_3,7} & 0 & -E_{1,7} & -E_{2,7} & 0 & 0   \\
   \mathbf{E_{1,3}} & 0 & 0 & 0 & 0 & 0 & 0 & 0 & E_{1,5'} & 0 & 0 & E_{1,6} & E_{1,7}   \\
   \mathbf{E_{2,3}} & 0 & 0 & 0 & 0 & 0 & 0 & 0 & E_{2,5'} & 0 & 0 & E_{2,6} & E_{2,7}  \\
  }$$
\caption{$C(\mathfrak{g}^{\prec}(\mathcal{P}_1))$}\label{fig:matpostmap}
\end{figure}

\noindent
In Figure~\ref{fig:posetmap} we illustrate the Hasse diagram of $\mathcal{P}$ \textup(left\textup) and the Hasse diagram of the height-two poset $\mathcal{P}_1$ \textup(right\textup).

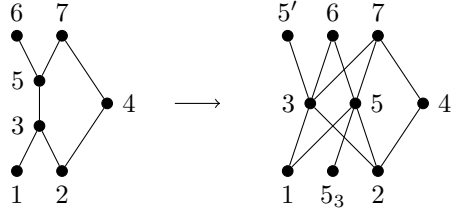
\begin{figure}[H]
$$\begin{tikzpicture}[scale=0.6]
\node (v1) at (-2.5,-1.5) [circle, draw = black, fill = black, inner sep = 0.5mm, label=below:{1}] {};
\node (v6) at (-1.5,-1.5) [circle, draw = black, fill = black, inner sep = 0.5mm, label=below:{2}] {};
\node (v2) at (-2,-0.5) [circle, draw = black, fill = black, inner sep = 0.5mm, label=left:{3}] {};
\node (v3) at (-2,0.5) [circle, draw = black, fill = black, inner sep = 0.5mm, label=left:{5}] {};
\node [circle, draw = black, fill = black, inner sep = 0.5mm, label=above:{6}] (v5) at (-2.5,1.5) {};
\node [circle, draw = black, fill = black, inner sep = 0.5mm, label=above:{7}] (v7) at (-1.5,1.5) {};
\node (v8) at (-0.5,0) [circle, draw = black, fill = black, inner sep = 0.5mm, label=right:{4}] {};
\draw (v1) -- (v2) -- (v3);
\draw (v3) -- (v5);
\draw (v6) -- (v2);
\draw (v3) -- (v7) -- (v8) -- (v6);
\draw[->] (1,0) -- (2,0);
\node [circle, draw = black, fill = black, inner sep = 0.5mm, label=below:{1}] (v9) at (3.5,-1.5) {};
\node [circle, draw = black, fill = black, inner sep = 0.5mm, label=below:{$5_3$}] (v11) at (4.5,-1.5) {};
\node [circle, draw = black, fill = black, inner sep = 0.5mm, label=left:{3}] (v14) at (4,0) {};
\node [circle, draw = black, fill = black, inner sep = 0.5mm, label=right:{5}] (v10) at (5,0) {};
\node [circle, draw = black, fill = black, inner sep = 0.5mm, label=above:{$5'$}] (v15) at (3.5,1.5) {};
\node [circle, draw = black, fill = black, inner sep = 0.5mm, label=above:{6}] (v13) at (4.5,1.5) {};
\node [circle, draw = black, fill = black, inner sep = 0.5mm, label=above:{7}] (v17) at (5.5,1.5) {};
\node [circle, draw = black, fill = black, inner sep = 0.5mm, label=right:{4}] (v18) at (6.5,0) {};
\node [circle, draw = black, fill = black, inner sep = 0.5mm, label=below:{2}] (v12) at (5.5,-1.5) {};
\draw (v9) -- (v10) -- (v11);
\draw (v12) -- (v10) -- (v13) -- (v14) -- (v9);
\draw (v12) -- (v14) -- (v15);
\draw (v10) -- (v17);
\draw (v14) -- (v17);
\draw (v17) -- (v18) -- (v12);
\end{tikzpicture}$$
\caption{Height-three poset reduced to height-two poset}\label{fig:posetmap}
\end{figure}
\end{Ex}

\section{Epilogue}\label{sec:conc}

Equation (\ref{eqn:mainthm}) provides a height-independent combinatorial formula for the index of a nilpotent Lie poset algebra. For the (solvable) Lie poset algebra case, the situation is more complicated. The best formula we have is given by the following recent theorem.


\noindent



\begin{theorem*}[\textbf{\cite{CM}}, 2019]
If $\mathcal{P}$ is a poset of height at most two, then 
$$\ind\mathfrak{g}(\mathcal{P})=|Rel_E(\mathcal{P})|-|\mathcal{P}|+2\cdot C_{\mathcal{P}}+\sum_{p\in \mathcal{P}\backslash Ext(\mathcal{P})}UD(\mathcal{P},p),$$ where $C_{\mathcal{P}}$ denotes the number of components in the Hasse diagram of $\mathcal{P}$ and

$$UD(\mathcal{P},p) =  \begin{cases} 
      |U(\mathcal{P},j)-D(\mathcal{P},p)|, & U(\mathcal{P},p)\neq D(\mathcal{P},p); \\
      2, & \text{otherwise.}
   \end{cases}
$$
\end{theorem*}

\begin{Ex}\label{ex:posetmat}
Let $\mathcal{P}$ be the poset given in Example~\ref{ex:posnot}; that is, $\mathcal{P}=\{1,2,3,4,5,6\}$ with $1,2\preceq 3\preceq 4,5,6$. The matrix form of elements in $\mathfrak{g}(\mathcal{P})$ is illustrated in Figure~\ref{fig:tA}, where the $*$'s denote potential non-zero entries. We have that $$|Rel_E(\mathcal{P})|=6,\quad|\mathcal{P}|=6,\quad C_{\mathcal{P}}=1,\quad \mathcal{P}\backslash Ext(\mathcal{P})=\{3\},\quad U(\mathcal{P},3)=3,\quad D(\mathcal{P},3)=2,\quad\text{and}\quad UD(\mathcal{P},3)=1.$$ Thus, $\ind\mathfrak{g}(\mathcal{P})=6-6+2+1=3$. Note that this differs from the index of the corresponding nilpotent Lie poset algebra $\mathfrak{g}^{\prec}(\mathcal{P})$, found in Example~\ref{ex:h2commat} to be 7.
\begin{figure}[H]
$$\kbordermatrix{
    & 1 & 2 & 3 & 4 & 5 & 6 \\
   1 & * & 0 & * & *  & * & *   \\
   2 & 0 & * & * & * & * & *  \\
   3 & 0 & 0 & * & * & * & *  \\
   4 & 0 & 0 & 0 & * & 0 & 0  \\
   5 & 0 & 0 & 0 & 0 & * & 0  \\
   6 & 0 & 0 & 0 & 0 & 0 & *  \\
  }$$
\caption{Matrix form of $\mathfrak{g}(\mathcal{P})$}\label{fig:tA}
\end{figure}
\end{Ex}

In moving from the nilpotent to the solvable case, the only change is the addition of basis elements corresponding to diagonal matrices. This modification is substantive -- such diagonal elements appear to form obstructions to applying an inductive argument, similar to that used here, to establish height-independent index formulas.


\end{document}